\documentclass{article}

\RequirePackage{amsmath,amsthm,amssymb,enumerate}

  \RequirePackage{tikz}
  \tikzstyle{dot}=[circle, inner sep=.4mm, draw=black, outer sep=2mm, fill=black]
  \newcommand{\plotptradius}{0.275}
  \newcommand{\setplotptradius}[1]{\renewcommand{\plotptradius}{#1}}
  \newcommand{\plotpt}[2][] 
  { \fill[#1,radius=\plotptradius] (#2) circle; }
  \newcommand{\plotperm}[2][black]  
  {
    \foreach \y [count=\x] in {#2}
      \plotpt[#1]{\x}{\y};
  }

  \newcommand{\X}{{\bf X}}

  \usepackage{mathpazo}
  \usepackage[compact]{titlesec}
  \titleformat{\section}{\Large \sffamily}{\thesection.}{0.5em}{}[\titlerule]
 \titleformat{\subsection}{\large \sffamily}{\thesubsection.}{0.5em}{}
  \RequirePackage{color}
  \definecolor{lightgray}{rgb}{0.8, 0.8, 0.8}
  \definecolor{darkgray}{rgb}{0.7, 0.7, 0.7}
  \definecolor{darkblue}{rgb}{0, 0, .4}
  \definecolor{red}{rgb}{.6, 0, 0}
  \RequirePackage[bookmarks]{hyperref}
  \hypersetup{colorlinks=true, linkcolor=darkblue, anchorcolor=darkblue,
          citecolor=darkblue, urlcolor=darkblue}

 \newcommand{\br}{d}
  
  \DeclareMathOperator{\mj}{\mathsf{mj}}
  \newcommand{\Sn}{\mathfrak{S}_n}
  \renewcommand{\Pr}[1]{\mathbb{P}\left\{#1\right\}}
  \newcommand{\Ex}[1]{\mathbb{E}\left[#1\right]}
  
 \newcommand{\sO}{\tilde{O}}

  \newtheorem{theorem}{Theorem}
  \newtheorem{lemma}[theorem]{Lemma}

\theoremstyle{definition}

\title{\sffamily The minimum Manhattan distance and minimum jump of permutations}
\author{\sffamily Simon R. Blackburn (corresponding author)\\
\sffamily Department of Mathematics\\
\sffamily Royal Holloway University of London\\
\sffamily Egham, Surrey TW20 0EX\\
{\tt s.blackburn@rhul.ac.uk}\\
\\
\sffamily Cheyne Homberger\\
\sffamily Department of Mathematics and Statistics\\
\sffamily University of Maryland, Baltimore County\\
\sffamily Baltimore, MD 21250, U.S.A.\\
{\tt cheyneh@umbc.edu}\\
\\
\sffamily Peter Winkler\thanks{This author is supported by NSF grants DMS-1162172 and DMS-1600116}\\
\sffamily Department of Mathematics\\
\sffamily Dartmouth, Hanover, NH 03755--3551, U.S.A.\\
{\tt peter.winkler@dartmouth.edu}
}

\date{\sffamily \today}

\begin{document}
\maketitle
\begin{abstract}
Let $\pi$ be a permutation of $\{1,2,\ldots,n\}$. If we identify a permutation with its graph, namely the set of $n$ dots at positions $(i,\pi(i))$, it is natural to consider the minimum $L^1$ (Manhattan) distance, $\br(\pi)$, between any pair of dots. The paper computes the expected value (and higher moments) of $\br(\pi)$ when $n\rightarrow\infty$ and $\pi$ is chosen uniformly, and settles a conjecture of Bevan, Homberger and Tenner (motivated by permutation patterns), showing that when $d$ is fixed and $n\rightarrow\infty$, the probability that $d(\pi)\geq d+2$ tends to $e^{-d^2 - d}$. 

The minimum jump $\mj(\pi)$ of $\pi$, defined by $\mj(\pi)=\min_{1\leq i\leq n-1} |\pi(i+1)-\pi(i)|$, is another natural measure in this context. The paper computes the asymptotic moments of $\mj(\pi)$, and the asymptotic probability that $\mj(\pi)\geq d+1$ for any constant $d$.
\end{abstract}



\newpage

\section{Introduction}

Let $n$ be a positive integer, with $n\geq 2$. We write $[n]$ for the set $\{1,2,\ldots,n\}$, and we write $\Sn$ for the set of all permutations of $[n]$. We write a permutation $\pi \in \mathfrak{S}_n$ in one-line notation, so $\pi = \pi(1) \pi(2) \ldots \pi(n)$. Recall that the \emph{graph} of a permutation $\pi$ is the set of points (\emph{dots}) of the form $(i,\pi(i))$ for $i\in [n]$. Figure~\ref{fig:perm_example} depicts the graph of the permutation $\pi=147258369\in\mathfrak{S}_9$.

The \emph{minimum Manhattan distance} $\br(\pi)$ of a permutation $\pi$ is defined by: 
\begin{equation}
\label{eqn:breadth_defn}
\br(\pi) = \min_{1\leq i<j\leq n}\{|i-j| + |\pi(i) - \pi(j)|\} .
\end{equation}
The permutation $\pi$ in Figure~\ref{fig:perm_example} has $\br(\pi)=4$ (which is, in fact, the largest possible value for a permutation in $\mathfrak{S}_9$). Note that for $n\geq 2$ we have $\br(\pi)\geq 2$ for all $\pi\in\Sn$.

The minimum Manhattan distance is a natural measure when thinking of a permutation as its graph, but was first studied (under the name of the \emph{breadth} of a permutation) by Bevan, Homberger, and Tenner~\cite{cheyne:prolific} in the context of permutation patterns. We now briefly explain this context.

\begin{figure}
  \centering
  \begin{tikzpicture}[scale=.5]
    \setplotptradius{2.5mm}
    \draw[dotted] (1, 1) grid (9, 9);
    \plotpt[black!50]{1,1};
    \plotpt[black!50]{2,4};
    \plotpt[black!50]{3,7};
    \plotpt[black!50]{4,2};
    \plotpt[black!50]{5,5};
    \plotpt[black!50]{6,8};
    \plotpt[black!50]{7,3};
    \plotpt[black!50]{8,6};
    \plotpt[black!50]{9,9};
  \end{tikzpicture}
  \caption{The graph of the permutation $\pi=147258369\in\mathfrak{S}_9$.}
  \label{fig:perm_example}
\end{figure}
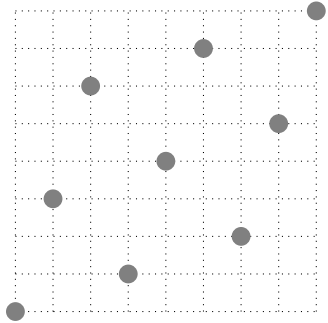

Two sequences $a = a_1, a_2, \ldots, a_n$ and $b = b_1, b_2, \ldots, b_n$ of
distinct real numbers are said to have the same \emph{relative order} if $a_i
< a_j$ precisely when $b_i < b_j$. For a given sequence $a$ of length $n$, we
define the \emph{standardization of $a$} to be the unique
sequence on the letters $[n]$ which is in the same
relative order as $a$. The pattern ordering imposes a partial order on the set of all permutations:
For $\pi\in\Sn$ and $\sigma \in \mathfrak{S}_k$, we say that
\emph{$\pi$ contains $\sigma$ as a pattern} (denoted $\sigma \prec \pi$) if
there is a subsequence of $\pi(1) \pi(2) \ldots \pi(n)$ whose
standardization is equal to $\sigma(1) \sigma(2) \dots \sigma(k)$. For
example, $213 \prec 34152$ as seen by the first, third, and fourth entries,
while $21 \not \prec 123456$. 

A permutation $\pi\in\Sn$ can contain
at most $n$ distinct patterns of length $n{-}1$, at most $n(n{-}1)/2$ patterns of
length $n{-}2$, and at most $\binom{n}{d}$ patterns of length $n{-}d$. For an
integer $d$, a permutation is \emph{$d$-prolific} if it contains
precisely $\binom{n}{d}$ distinct patterns of length $n{-}d$. Equivalently, a
permutation is $d$-prolific if every choice of $d$ deletions yields a different
pattern. The notion of a prolific permutation was introduced by
Homberger~\cite{cheyne:fixed_patterns}; see B\'ona~\cite{bona:book} for a general survey of results in permutation patterns. Recently, Bevan, Homberger, and
Tenner~\cite[Theorem~2.22]{cheyne:prolific} established a tight connection between the prolific property and minimum Manhattan distance: a permutation $\pi$ is $d$-prolific if and only if $\br(\pi) \geq d{+}2$. This connection provides another motivation for the study of the minimum Manhattan distance of a permutation.

In a 1944 paper, Wolfowitz~\cite{wolfowitz:runs} proved that a random
permutation has no pairs of dots at distance 2 with probability tending to $e^{-2}$, which
implies that the number of $1$-prolific permutations of length $n$ is
asymptotic to $n!/e^2$.  Bevan et al.\ conjectured that a more general result was true: the proportion of
$d$-prolific permutations should be asymptotically equal to $e^{-d^2-d}$.  The results of random trials seem to confirm this: Figure~\ref{fig:trials} summarises the results of calculating $\br(\pi)$ for $10\,000\,000$ permutations $\pi$, and compares these results to the values predicted by Bevan et al.

\begin{figure}
\[
\begin{array}{l|llll}
n\backslash \text{min distance}&2&3&4&5\\\hline
100&8647017&1330004&22931&48\\
200&8646302&1329900&23751&47\\
400&8646991&1328733&24213&63\\
600&8646229&1329301&24420&50\\
800&8645241&1330035&24660&64\\
1000&8649121&1326350&24466&63\\
2000&8646825&1328391&24722&62\\\hline\hline
\text{Predicted}&8646647&1328565&24726& 61
\end{array}
\]
\caption{The minimum distance of $10\,000\,000$ permutations of degree $n$.}
\label{fig:trials}
\end{figure}

We will prove the following theorem, using a proof framework due to Aspvall and Liang~\cite{aspvall}:

\begin{theorem}
\label{thm:direct}
Let $d$ be a function of $n$ such that $d=O(\log n)$. Define $\lambda=d^2+d$. The probability that a uniformly chosen permutation $\pi\in\Sn$ is $d$-prolific (in other words, the probability that $\br(\pi)\geq d+2$) is $e^{-\lambda}+O((\log n)^6e^{\lambda}/n)$.
\end{theorem}
\noindent (We emphasise that the implied constant in the error term in Theorem~\ref{thm:direct} depends only on the implied constant in our bound for $d$.) We note that when $d$ is constant, $(\log n)^6 e^{\lambda}/n\rightarrow 0$ as $n\rightarrow\infty$ and so the conjecture of Bevan et al.\  follows from Theorem~\ref{thm:direct}.

Wolfowitz's theorem has another natural generalisation. The \emph{minimum jump} $\mj(\pi)$ of a permutation $\pi$ is defined as
\[
\mj(\pi)=\min_{1\leq i\leq n-1} |\pi(i+1)-\pi(i)|.
\]
So $\mj(\pi)\geq 1$ for all $\pi\in\Sn$, and the permutation $\pi$ in Figure~\ref{fig:perm_example} has $\mj(\pi)=3$. Wolfowitz's theorem can be thought of as asserting that the probability that a uniformly chosen permutation $\pi$ has $\mj(\pi)\geq 2$ tends to $e^{-2}$ as $n\rightarrow\infty$. We will prove the more general statement: the probability that a uniformly chosen permutation $\pi$ has $\mj(\pi)\geq k+1$ tends to $e^{-2k}$. The framework of Aspvall and Liang is, in fact, good enough to prove this result, but we work significantly harder (introducing new techniques) to obtain the following stronger theorem, which gives the form of the lower order terms of the asymptotics:

\begin{theorem}
\label{thm:mj_prob}
Let $t$ be a fixed positive integer. There exist functions
\[
p_1(x), p_2(x),\ldots,p_{t-1}(x),q(x)
\]
of $x$, all bounded in absolute value by a polynomial in $x$, with the following property. Let $d$ be a function of $n$ such that $d=O(\log n)$. Let the permutation $\pi\in\Sn$ be chosen uniformly at random. Then
\[
\Pr{\mj(\pi)\geq d+1}=\left(1+\sum_{i=1}^{t-1}p_i(d)/n^{i}\right)e^{-2d}+O(q(d)e^{2d}/n^t).
\]
\end{theorem}

The extra asymptotic information from Theorem~\ref{thm:mj_prob} is a key ingredient in the proofs of Theorems~\ref{thm:expectedbreadth} and~\ref{thm:expectedjump} below.

Let $\mathbf{Y}$ be the integer-valued random variable where, for any non-negative integer $d$,
\[
\Pr{\mathbf{Y}\geq d+2}=e^{-d^2-d}.
\]
So for $d\in\mathbb{Z}$
\[
\Pr{\mathbf{Y}=d+2}=\begin{cases} 
0&\text{when }d<0\\
e^{-d^2-d}-e^{-(d+1)^2-(d+1)}&\text{when }d\geq 0.
\end{cases}
\]
For an integer $n$ with $n\geq 2$, let $\mathbf{Y}_n$ be the random variable $\br(\pi)$ where $\pi\in\Sn$ is chosen uniformly. Theorem~\ref{thm:direct} shows that the sequence $\mathbf{Y}_2,\mathbf{Y}_3,\ldots$ of random variables converges in distribution to $\mathbf{Y}$. In fact, more is true:

\begin{theorem}
\label{thm:expectedbreadth}
Let $a$ be a fixed non-negative integer. Let the random variable $\mathbf{Y}$ and the random variables $\mathbf{Y}_n$ be defined as above. The $a$th moment of the sequence $\mathbf{Y}_2,\mathbf{Y}_3,\ldots$ converges to the $a$th moment of $\mathbf{Y}$. In particular, the following is true. 
Let $\pi\in\Sn$ be chosen uniformly. The expected value of the minimum Manhattan distance $\br(\pi)$ of $\pi$ tends to 
  \[ 1 + \sum_{d=0}^\infty e^{-d^2-d} \approx 2.1378201816868795778 \]
  as $n\rightarrow\infty$.
\end{theorem}

Note that the information on the asymptotic distribution of $\br(\pi)$ we establish in Theorem~\ref{thm:direct} is not, by itself, sufficient to find the expected value of $\br(\pi)$, as error terms are not tight enough when $\br(\pi)$ is large. (Indeed, Theorem~\ref{thm:direct} is vacuous when $d\sim\log n$, for example.) So instead we use the better asymptotics we obtain for $\mj(\pi)$, together with the obvious fact that $\mj(\pi)\leq \br(\pi)+1$, to provide sufficiently good approximations when $d$ is large. 

A similar result holds for the distribution of the minimum jump of a permutation. Define $\mathbf{Z}$ to be the integer valued random variable such that $\Pr{\mathbf{Z}\geq d+1}=e^{-2d}$ for any non-negative integer $d$. For an integer $n$, let $\mathbf{Z}_n$ be the random variable $\mj(\pi)$ where $\pi\in\Sn$ is chosen uniformly. Theorem~\ref{thm:mj_prob} shows that the sequence $\mathbf{Z}_n$ of random variables converges in distribution to $\mathbf{Z}$ as $n\rightarrow\infty$. The following theorem is the analogue of Theorem~\ref{thm:expectedbreadth}: 

\begin{theorem}
\label{thm:expectedjump}
Let $a$ be a fixed non-negative integer. Let the random variable $\mathbf{Z}$ and the random variables $\mathbf{Z}_n$ be defined as above. The $a$th moment of the sequence $\mathbf{Z}_1,\mathbf{Z}_2,\ldots$ converges to the $a$th moment of $\mathbf{Z}$. In particular, the following is true. 
Let $\pi\in\Sn$ be chosen uniformly. The expected value of the minimum jump $\mj(\pi)$ of $\pi$ tends to 
  \[ \sum_{d=0}^\infty e^{-2d} \approx 1.1565176427496656518 \]
  as $n\rightarrow\infty$.
\end{theorem}

Besides playing a role in the study of the asymptotic moments of $\br(\pi)$, the study of minimum jumps is also useful to analyse an efficient algorithm to compute the minimum distance of a permutation $\pi$, which we describe as follows.

The naive algorithm to compute the minimum distance of a permutation takes $O(n^2)$ arithmetic operations, using~\eqref{eqn:breadth_defn} in a straightforward way. We may obtain an $O(n^{1.5})$ algorithm by
making use of the fact~\cite{cheyne:prolific} that 
$\br(\pi)\leq y+2$, where $y$ is the largest integer such that $y^2/2+2y+1\leq n$. Since pairs of dots at horizontal distance more than $y-1$ cannot contribute to the minimum in~\eqref{eqn:breadth_defn}, we may compute $\br(\pi)$ as
\begin{equation}
\label{eqn:better_formula}
\br(\pi)=\min_{\substack{1\leq i<j\leq n\\|i-j|< y}} |i-j|+|\pi(i)-\pi(j)|.
\end{equation}
Because $y=O(\sqrt{n})$, we only examine $O(n^{1.5})$ pairs $(i,j)$ in~\eqref{eqn:better_formula} and so we do indeed obtain an algorithm with $O(n^{1.5})$ complexity. A better algorithm has the same worst case complexity, but has an expected complexity of $O(n)$ when the permutation $\pi$ is chosen uniformly at random. The algorithm first computes $\mj(\pi)$, which can straightforwardly be done using $O(n)$ operations. Since $\mj(\pi)+1\geq \br(\pi)$, we may compute $\br(\pi)$ using~\eqref{eqn:better_formula} by setting $y=\mj(\pi)+1$ if this produces a smaller value for $y$. Theorem~\ref{thm:expectedjump} shows that the expected value of $\mj(\pi)$ is constant (and, in fact, small), and so this algorithm has $O(n)$ complexity on average. This algorithm performs well in practice, and was used to produce the data in Figure~\ref{fig:trials}.

The structure of the remainder of the paper is as follows. In Section~\ref{sec:sketch}, we provide an overview of the proofs in the paper.  We prove Theorem~\ref{thm:direct} in Section~\ref{sec:direct}. In Section~\ref{sec:jumps} we establish Theorem~\ref{thm:mj_prob} modulo a counting lemma, which we prove in Section~\ref{sec:sets}. Finally, in Section~\ref{sec:expectation}, we use Theorems~\ref{thm:direct} and~\ref{thm:mj_prob} to establish Theorems~\ref{thm:expectedbreadth} and~\ref{thm:expectedjump}.

\section{An overview of the proofs}
\label{sec:sketch}

Fix a positive integer $n$, and let $\pi$ be a random permutation of length $n$ chosen uniformly
from $\mathfrak{S}_n$. For an integer $d$ with $d<n$, let $\lambda = d(d{+}1)$. To prove Theorem~\ref{thm:direct}, we are interested in estimating the probability that $\br(\pi)\geq d+2$.

There are approximately $n^2\lambda$ positions $(x,y),(x',y')\in [n]\times [n]$ with $x<x'$ and $y\not=y'$ such that $|x-x'|+|y-y'|\leq d+1$. To see this, note that there are $n^2$ choices for $(x,y)$ and then, provided $x$ and $y$ are not too close to $n$, there are $2\binom{d+1}{2} =\lambda$ choices for $(x',y')$ (see Figure~\ref{fig:bad_pairs} below, where the hollow dot is the position $(x,y)$ and the filled dots are the potential positions for $(x',y')$). A position in $[n]\times [n]$ is occupied by a dot in the graph of $\pi$ with probability $1/n$, so the probability of a `bad' event, that both of $(x,y)$ and $(x',y')$ are occupied by a dot of $\pi$, should be approximately $1/n^2$. So it is reasonable to assume that the expected number of bad events should be close to $\lambda$, and that the distribution of the number of bad events should be approximately Poisson with mean $\lambda$. The probability that $\br(\pi)\geq d+2$ is exactly the probability that there are no bad events, and so this leads to the belief that the probability that $\br(\pi)\geq d+2$ should be $e^{-\lambda}$. The problem with this approach is that the bad events are not independent, and so this argument cannot be made rigorous as it stands. So we proceed as follows. 

For
a permutation $\pi$ of length $n$ and for two indices $i,j \in [n]$, we define
their \emph{distance} to be 
\[ \delta_\pi(i, j) = |i-j| + |\pi(i) - \pi(j)|. \]
We say that $i,j$ form a \emph{close pair} if $d_\pi(i,j) < d{+}2$. To prove Theorem~\ref{thm:direct}, we wish to determine the probability that $\pi$ has no close pairs. 

For each $i \in [n-1]$, define $\X_i$ to be indicator variable for the event that
the index $i$ is the first element of a close pair, i.e. the event that $(i,j)$ is a close pair
for some $j$.   (For $i<n-d$ the expected value of $\X_i$ should be approximately $\lambda/n$, since once $\pi(i)$ is fixed there are $\lambda$ positions where a dot would lead to a close pair: see Figure~\ref{fig:bad_pairs}.) For $I\subseteq [n-1]$, define $\X_I=\prod_{i\in I} \X_i$. For an integer $m$, define $S_m=\sum \Ex{\X_I}$, where the sum is over all subsets $I$ of cardinality $m$. The principle of inclusion-exclusion shows that the probability that all the variables $\X_i$ are zero (which is the probability that $\br(\pi)\geq d+2$) is $\sum_{m=0}^{n-1}(-1)^mS_m$. We will show that $S_m\approx \lambda^m/m!$ by considering terms $\Ex{\X_I}$ where $I$ is \emph{regular}, namely a subset $I\subseteq [n-d]$ such that $|i-i'|>d$ for all $i,i'\in I$, separately. We provide good estimates for $\Ex{\X_I}$ when $I$ is regular, and then show that $S_m$ is dominated by these values (as there are few irregular sets, and since $\Ex{\X_I}$ is never too large). This will establish Theorem~\ref{thm:direct}.
  
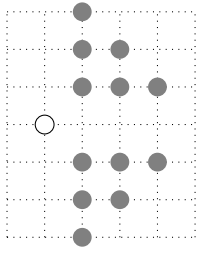
\begin{figure}
  \centering
  \begin{tikzpicture}[scale=.5]
    \setplotptradius{2.5mm}
    \draw[dotted] (-1, -3) grid (4, 3);
    \node[circle, fill=white, draw=black,inner sep=0, minimum size=\plotptradius] at (0,0) {};
    \plotpt[black!50]{1,1};
    \plotpt[black!50]{1,-1};
    \plotpt[black!50]{1,2};
    \plotpt[black!50]{1,-2};
    \plotpt[black!50]{2,1};
    \plotpt[black!50]{2,-1};
    \plotpt[black!50]{1,3};
    \plotpt[black!50]{1,-3};
    \plotpt[black!50]{2,2};
    \plotpt[black!50]{2,-2};
    \plotpt[black!50]{3,1};
    \plotpt[black!50]{3,-1};
  \end{tikzpicture}
  \caption{The filled dots mark the $2\binom{4}{2} = 12$ potential entries
  which will lead to the hollow dot starting a close pair with $d=3$.}
  \label{fig:bad_pairs}
\end{figure}

To prove Theorem~\ref{thm:mj_prob} we take a similar approach, but the random variables $\X_i$ are now indicator random variables for the event that $|\pi(i+1)-\pi(i)|\leq d$. We introduce new techniques, introducing various types of `regular' subsets $I$ in our estimates, in order to provide better error terms than the straightforward analogue of Theorem~\ref{thm:direct} would provide. These better error terms allow us to prove Theorems~\ref{thm:expectedbreadth} and~\ref{thm:expectedjump} in the final section of the paper.

\section{Proof of Theorem~\ref{thm:direct}}
\label{sec:direct}

We use the notation from the previous section, so we choose permutations $\pi\in\Sn$ uniformly and, for each $i\in [n-1]$, we define $\X_i$ to be the indicator variable for the event that the index $i$ is the first element of a close pair for $\pi$. For $I\subseteq [n-1]$, define $\X_I=\prod_{i\in I} \X_i$. The Principle of Inclusion--Exclusion implies that
\begin{equation}
\label{eqn:PIE}
\Pr{\pi\in\Sn\text{ is $d$-prolific}}=\sum_{m=0}^{n-1} (-1)^m S_m,
\end{equation}
where
\begin{equation}
\label{eqn:bin_moment_def}
S_m=\sum_{\substack{I\subseteq [n-1]\\ |I|=m}}\Ex{\X_I}.
\end{equation}
Indeed, the Bonferroni inequalities~\cite[Chapter~1]{galambos} (a generalisation of inclusion-exclusion) imply that partial sums of the right hand side of~\eqref{eqn:PIE} are successively upper and lower bounds for the left hand side. More precisely, for any $r$ such that $0\leq r<n-1$, we have that
\[
\Pr{\pi\in\Sn\text{ is $d$-prolific}}=\left(\sum_{m=0}^{r} (-1)^m S_m\right)+\epsilon
\]
where $|\epsilon|\leq S_{r+1}$. 

We claim that
\begin{equation}
\label{eqn:bin_moment_bound}
S_m=\big(1+O(d^2m^2/n)\big)\frac{\lambda^m}{m!},
\end{equation}
where the implied constant is absolute. Once we have established this claim, we may prove the theorem as follows. Let $r=\max\{12d^2,\lceil (\log_2 n)^2\rceil\}$. We note that $r=O((\log n)^2)$, and
\[
\frac{\lambda^{r+1}}{(r+1)!}\leq \left(\frac{e\cdot 2d^2}{r+1}\right)^{r+1}\leq 2^{-r}=o(1/n).
\]
Hence $S_{r+1}=O((\log n)^6e^{\lambda}/n)$, by~\eqref{eqn:bin_moment_bound}. Moreover, again using~\eqref{eqn:bin_moment_bound},
\begin{align*}
\sum_{m=0}^{r} (-1)^m S_m &= \sum_{m=0}^r (-1)^m \frac{\lambda^m}{m!}+O\left(\sum_{m=0}^r d^2m^2/n \frac{\lambda^m}{m!}\right)\\
&=\sum_{m=0}^r (-1)^m \frac{\lambda^m}{m!}+O\left(d^2r^2e^\lambda/n \right).
\end{align*}
Taylor's theorem, applied to the Maclaurin series for $e^{-x}$, shows that
\[
\sum_{m=0}^r (-1)^m \frac{\lambda^m}{m!}=e^{-\lambda}+O\left((\log n)^6e^\lambda/n \right),
\]
and so the theorem follows.

It remains to prove our claim~\eqref{eqn:bin_moment_bound}. 

To prove the upper bound in~\eqref{eqn:bin_moment_bound}, we first note that there are $\binom{n-1}{m}$ choices for the subset $I$ in~\eqref{eqn:bin_moment_def}. Once $I=\{i_1,i_2,\ldots,i_m\}$ is fixed, we observe that the permutations $\pi$ that contribute to the event that $\X_I=1$ may all be constructed in the following manner.

First choose $m$ integer vectors $(\gamma_j,\delta_j)$ (for $j\in [m]$) of (Manhattan) length at most $d$ and with $\gamma_j$ positive.  There are $\lambda^m$ choices for such vectors. We will restrict our attention to those permutations $\pi$ such that $\pi(i_j+\gamma_j)=\pi(i_j)+\delta_j$ for all $j\in[m]$, so our vectors are the relative positions of close pairs.

Define a graph $\Gamma$ with vertex set $[n]$ and $m$ edges $\{i_j,i_j+\gamma_j\}$. If $\Gamma$ is not well defined (because $i_j+\gamma_j>n$ for some $j$) there are no permutations of the form we are counting. Otherwise, we proceed as follows. Since the elements $i_j$ are distinct, the graph $\Gamma$ is a forest. So, since $\Gamma$ has $m$ edges and $n$ vertices, it has exactly $n-m$ components. Order the components $C_1,C_2,\ldots,C_{n-m}$ of $\Gamma$ from largest to smallest. Let $c_1,c_2,\ldots,c_{m-n}$ be the number of vertices in $C_1,C_2,\ldots,C_{n-m}$ respectively. Choose a vertex $s_\ell$ in each component in some way (by, for example, picking the smallest). We now choose $\pi(s_1),\pi(s_2),\ldots,\pi(s_{n-m})$ in turn; the permutation $\pi$ is then completely determined by these choices, by the definition of $\Gamma$ and by the fact that the vectors $(\gamma_i,\delta_i)$ are fixed. There are at most $n$ choices for $\pi(s_1)$, and once we have made this choice the values of $\pi(v)$ with $v\in C_1$ are determined. More generally, there are at most $n-\sum_{\ell=1}^{j-1} c_i$ choices for $\pi(s_j)$, and so the number of choices is
\[
\prod_{j=1}^{n-m}\left(n-\sum_{\ell=1}^{j-1}c_i\right)\leq \left(\prod_{j=0}^{m-1}(n-2j)\right)(n-2m)!.
\]
Thus
\[
S_m\leq \binom{n-1}{m}\frac{1}{n!}\lambda^m\left(\prod_{j=0}^{m-1}(n-2j)\right)(n-2m)!=\big(1+O(d^2m^2/n)\big)\frac{\lambda^m}{m!}.
\]

We now prove the lower bound in~\eqref{eqn:bin_moment_bound}. The summands in~\eqref{eqn:bin_moment_def} are all non-negative, so we may restrict our attention to those subsets $I=\{i_1,i_2,\ldots,i_m\}$ such that $1\leq i_1<i_2\cdots<i_m\leq n-d$ and $i_{j+1}-i_j>d$ for $j\in[m-1]$. The number of subsets $I$ of this form is at least $(n-(2d+1)(m+1))^m/m!$. We provide a lower bound on $\Ex{\X_I}$ by providing a lower bound on the number of permutations $\pi$ constributing to the event that $\X_I=1$.

We first choose the $m$ values $\pi(i_j)$ with $j\in [m]$ so that these values differ by at least $2d+1$ and such that $d\leq \pi(i_j)\leq n-d$. The number of choices for these values is at least $(n-(2d+1)(m+1))^m$. We choose $m$ integer vectors $(\gamma_j,\delta_j)$ of (Manhattan) length at most $d$ and with $\gamma_j$ positive. There are $\lambda^m$ choices for such vectors. We define $\pi(i_j+\gamma_j)=\pi(i_j)+\delta_j$. The restrictions we have placed on the set $I$ and the values $\pi(i_j)$ ensure that we have defined $2m$ distinct values in the image of $\pi$, and so we have specified a partial permutation. 

For $j\in\{1,2,\ldots ,m\}$ and $\gamma\in\{1,2,\ldots,d\}\setminus \{\gamma_j\}$, choose the value $\pi(i_j+\gamma)$ so that it lies at distance at least $d+2$ from the $\pi(i_j)$ and is distinct from all previous choices. There are at least
\[
(n-(2d+1)-2m-(d-1)m)^{(d-1)m}\geq (n-(2d+1)(m+1))^{(d-1)m}
\]
choices for these values. The remaining $n-2m-(d-1)m$ values $\pi(i)$ are chosen in an arbitrary fashion to complete the permutation: there are $(n-2m-(d-1)m)!$ choices for these remaining values.

All permutations we have constructed have the property that for all $i\in I$ there is a unique dot to the right of the $i$th dot at distance at most $d+1$, so all permutations we have constructed are distinct. Since each permutation occurs with probability $1/n!$, we find that
\[
S_m\geq \frac{\lambda^m}{m!}\frac{\left(n-(2d+1)(m+1)\right)^{(d+1)m}}{n(n-1)\cdots (n-(d+1)m+1)}\geq \big(1+O(d^2m^2/n)\big)\frac{\lambda^m}{m!}.
\]
So our claim~\eqref{eqn:bin_moment_bound} follows, and the theorem is proved.\hfill\qed

Before continuing, we note that this proof can be adapted to find the asymptotic distribution for the number of close pairs in a random permutation when $d$ is fixed:

\begin{theorem}
Let $d \geq 1$ be fixed. As $n\rightarrow \infty$, for a uniformly chosen permutation $\pi$, the number of pairs $i,j$ with $\delta_\pi(i,j) \leq d+1$ approaches the Poisson distribution with mean $\lambda = d^2 + d$.
\end{theorem}
\begin{proof}
Let $Y = \sum_{i=1}^{n-1} X_i$ be the number of entries which start a close pair. Note that, for $m \geq 0$, we have
\[ \Ex{\binom{Y}{m}} = S_m. \]
From this, it follows that $m! S_m$ is equal to the $m$th factorial moment $\Ex{(Y)_m}$.

From \eqref{eqn:bin_moment_bound} above, we have that $m!S_m$ converges to $\lambda^m$ as $n \rightarrow \infty$, which is precisely the $m$th factorial moment of the Poisson distribution with mean $\lambda$.  Since the Poisson distribution is determined by its moments~\cite{von_mises} (and hence by its factorial moments), it follows that $Y$ is asymptotically Poisson. But it is not hard to see that the number of indices of a uniformly chosen permutation that are at the start of more than one close pair is $O(1/n^2)$, and so  the number of close pairs is well approximated by $Y$.
\end{proof}

\section{Minimum jumps}
\label{sec:jumps}

The aim in this section and the next is to prove Theorem~\ref{thm:mj_prob}. In this section, we prove the theorem modulo a key lemma (Lemma~\ref{lem:final_counting}, stated in the proof below) which counts the number of sequences of distinct integers in $[n]$ where some pairwise differences are constrained to be small. In Section~\ref{sec:sets}, we prove Lemma~\ref{lem:final_counting}, thus establishing Theorem~\ref{thm:mj_prob}.

\begin{proof}[Proof of Theorem~\ref{thm:mj_prob}]
Let $K$ be the set of non-zero integers $\delta$ with $|\delta|\leq d$. So $\mj(\pi)\geq d+1$ if and only if $\pi(i+1)-\pi(i)\not\in K$ for all $i\in [n-1]$. For $i\in [n-1]$, we define $\X_i$ to be the indicator variable for the event that $\pi(i+1)-\pi(i)\in K$, when $\pi$ is chosen uniformly from $\Sn$. For $I\subseteq [n-1]$, define $\X_I=\prod_{i\in I} \X_i$ and define
\[
S_m=\sum_{\substack{I\subseteq [n-1]\\ |I|=m}}\Ex{\X_I}.
\]

Let $r=\max\{12d,\lceil (t+1)\log_2 n\rceil \}$.
As in Section~\ref{sec:direct}, the Bonferroni inequalities imply that (when $n$ is sufficiently large, so $r<n-1$)
\begin{equation}
\label{eqn:mj_bonferroni}
\Pr{\mj(\pi)\geq d+1}=\left(\sum_{m=0}^{r} (-1)^m S_m\right)+\epsilon
\end{equation}
where $|\epsilon|\leq S_{r+1}$.

Our theorem will follow once we have established good estimates on $S_m$. More precisely, our estimates can be written as follows. We say that a function $p$ of $d$ and $m$ is \emph{$d$-small polynomial in $m$}, or $d$-SP($m$), if $p$ can be written as a polynomial in $m$ with coefficients (which are functions of $d$ only) bounded in absolute value by a polynomial in $d$. We will show that there exist $d$-SP($m$) functions $\overline{p}_0,\overline{p}_1,\ldots, \overline{p}_{t-1},\overline{q}$ such that for all $m\leq r+1$
\begin{equation}
\label{eqn:core_estimate}
S_m=\frac{(2d)^m}{m!}\left(\sum_{i=0}^{t-1}\overline{p}_i/n^i+O(\overline{q}/n^t)\right).
\end{equation}
Moreover, we find that $\overline{p}_0=1$.

We remark that the fact that the functions $\overline{p}_i$ and $\overline{q}$ are polynomials in $m$ (rather than just being bounded by polynomials) is important, because this enables us to use results in the theory of exponential polynomials~\cite{boyadzhiev} to prove estimates~\eqref{eqn:error_terms} and~\eqref{eqn:main_terms} below. 

Suppose we have shown~\eqref{eqn:core_estimate} holds. Our choice of $r$, together with the fact that $(r+1)!\geq ((r+1)/e)^{r+1}$, shows that $(2d)^{r+1}/(r+1)!\leq n^{-(t+1)}$. Moreover, $\overline{p_i}=o(n)$ and $\overline{q}=o(n)$ since $d=O(\log n)$ and $r=O(\log n)$. So~\eqref{eqn:core_estimate} shows that $|\epsilon|\leq S_{r+1}=o(n^{-t})$. Hence $\Pr{\mj(\pi)\geq d+1}$ is well approximated by $\sum_{m=0}^{r} (-1)^m S_m$. So the theorem will now follow from equations~\eqref{eqn:mj_bonferroni} and~\eqref{eqn:core_estimate}, once we have established firstly that
\begin{equation}
\label{eqn:error_terms}
\sum_{m=0}^{r}\frac{(2d)^m}{m!}\left|\overline{q}\right|=O(q_0e^{2d})
\end{equation}
for some polynomially bounded function $q_0$ of $d$, and secondly that, for any fixed $i$ with $0\leq i\leq t-1$,
\begin{equation}
\label{eqn:main_terms}
\sum_{m=0}^{r}(-1)^m\frac{(2d)^m}{m!}\overline{p}_i=p_i e^{-2d}+o(n^{-t}e^{2d}),
\end{equation}
for some polynomially bounded function $p_i$ of $d$.

To prove~\eqref{eqn:error_terms}, let $a_1$ and $a_2$ be positive integers such that $\left|\overline{q}\right|\leq (md)^{a_1}+a_2$. Let $b(x)$ be the real polynomial such that for each $m\geq 0$ the $m$th derivative of $b(x)e^x$ is equal to $m^{a_1}$ when evaluated at zero. (The polynomial $b(x)$ is known as the $a_1$th \emph{exponential polynomial}; its coefficients are given by Stirling numbers of the second kind. See the survey by Boyadzhiev~\cite{boyadzhiev}, especially equations~(2.9) and~(3.4). Boyadzhiev attributes these results to Grunert~\cite{grunert} in 1843.) Using the Maclaurin series of $b(x)e^x$, we find that
\[
\sum_{m=0}^{r}\frac{(2d)^m}{m!}\left|\overline{q}\right|\leq \sum_{m=0}^{\infty}\frac{(2d)^m\left(d^{a_1}m^{a_1}+a_2\right)}{m!}=(d^{a_1}b(2d)+a_2)e^{2d}.
\]
So~\eqref{eqn:error_terms} holds with $q_0=d^{a_1}b(2d)+a_2$.

To prove~\eqref{eqn:main_terms}, let $b_i(x)$ be the polynomial such that for each $m\geq 0$ the $m$th derivative of $b_i(x)e^x$ is equal to $\overline{p}_i$ when evaluated at zero. (This polynomial may be constructed by taking an appropriate combination of the exponential polynomials mentioned above.) Note that the coefficients of $b_i(x)$ are functions of $d$. Indeed, $b_i(x)$ is a $d$-SP($x$) function, since $\overline{p}_i$ is a $d$-SP($m$) function. Taylor's inequality now implies that
\[
\sum_{m=0}^{r}(-1)^m\frac{(2d)^m}{m!}\overline{p}_i=b_i(-2d)e^{-2d}+O\left( \frac{M (-2d)^{r+1}}{(r+1)!} \right),
\]
where $M$ is the maximum value of the $r+1$-st derivative of $b_i(x)e^x$ in the real interval $[-2d,2d]$. The $r+1$-st derivitive is of the form $f(x)e^x$ for some polynomial $f(x)$. The degrees of $f(x)$ and $b_i(x)$ are equal (equal to $k$, say), and the coefficient of $x^{k-j}$ in $f(x)$ is a sum of the coefficients of $b_i(x)$ multiplied by polynomials in $r$ of degree at most $j$.  So $M=o(n e^{2d})$, since $b_i(x)$ is a $d$-SP($x$) function, and since $r$ and $d$ are logarithmic in $n$. Moreover $|(-2d)^{r+1}/(r+1)!|=(2d)^{r+1}/(r+1)!=O(n^{-(t+1)})$, and so~\eqref{eqn:main_terms} holds with $p_i=b_i(-2d)$. Finally, note that $b_0(x)=1$ and so $p_0=1$, as claimed.

To prove the theorem it remains to establish the bounds~\eqref{eqn:core_estimate} on $S_m$.

Recall that $S_m$ is a sum of terms $\Ex{\X_I}$, indexed by subsets $I\subseteq [n-1]$ of size $m$. In our proof of the analogous theorem for the minimum Manhattan distance, our bounds on $S_m$ used the notion of regular subsets $I$, those subsets whose elements are widely spaced in the interval. Since few sets are irregular, and since $\Ex{\X_I}$ is never too large, it was only the regular sets that contributed significantly to the sum; and it was possible to estimate $\Ex{\X_I}$ precisely for regular sets. We take a similar approach here. If we were only interested in estimates for the leading term of $S_m$ we could define a regular set to be a set $I$ with no pair of adjacent elements, and then bound $S_m$ in the same way as before. But we are interested in lower order terms of $S_m$, and so this approach is not quite sufficient. Roughly speaking, we proceed as follows. We associate an integer $z$ with every subset $I$, corresponding to the number of adjacent elements in $I$, and say that $I$ is \emph{$z$-irregular}. The $0$-irregular sets correspond to the regular sets mentioned above, and we will show that the leading $t$ terms of $S_m$ are determined by the $z$-irregular sets with $z\leq t-1$. In fact, the detailed asymptotic behaviour of $I$ is determined by more than just $z$, so we further partition $z$-irregular sets into types, indexed by the integer partitions $\lambda$ of $z$, and provide bounds on the $z$-regular sets of type $\lambda$ separately.

More precisely, we define the irregularity and type of a subset $I\subseteq [n-1]$ as follows. A subset $I$ with $|I|=m$ is made up of $c$ runs of consecutive elements, where $1\leq c\leq m$. The lengths $\ell_1,\ell_2,\ldots,\ell_c$ of these runs sum to $m$.
We say that the \emph{type} of a set $I$ is the integer partition $\lambda$ whose parts are precisely those integers $\ell_i-1$ that are positive. We say that $I$ is \emph{$z$-irregular} when $\lambda$ is a partition of $z$. We note that $z=m-c$. For example, the subset $I=\{3,4,6,7,8,10\}$ has $c=3$ runs of lengths $2$, $3$ and $1$ respectively, so is $3$-regular of type $\lambda=[2,1]$.

Now
\begin{equation}
\label{eqn:split_sum}
S_m
=\sum_{z=0}^{t-1}\sum_{\substack{|I|=m\\
\text{$z$-irregular}}}\Ex{\X_I}+\sum_{z\geq t}\sum_{\substack{|I|=m\\
\text{$z$-irregular}}}\Ex{\X_I}.
\end{equation}
We will see below that the first sum on the right hand side of~\eqref{eqn:split_sum} contributes the most significant terms, and the second sum is relatively small. So we will count the first sum carefully, and then provide a crude upper bound for the second sum.

Let $\lambda$ be a partition of the integer $z$. We may compute the number of subsets $I\subseteq [n-1]$ of type $\lambda$ with $|I|=m$ as follows. Let $u_i$ be the number of parts of $\lambda$ equal to $i$, so $\lambda=[1^{u_1}2^{u_2}\cdots z^{u_z}]$. An $m$-subset $I$ of type $\lambda$ is made up of $c=m-z$ runs of lengths $\ell_1,\ell_2,\ldots,\ell_c$; the number of choices for these lengths is the multinomial coefficient $\binom{m-z}{u_1,u_2,\ldots,u_z}$. Once these lengths are fixed, $I$ is determined by the lengths of the gaps between runs. The gap lengths are a sequence of $c+1$ non-negative integers that sum to $n-1-m$, with all integers strictly positive apart from possibly the first or last. So the number of choices for these gap lengths is $\binom{n-m}{m-z}$. Hence, writing $\rho:=\sum_{i=0}^zu_i$ for the number of parts of $\lambda$, the number of subsets of $I\subseteq [n-1]$ of type $\lambda$ with $|I|=m$ is 
\[
\binom{m-z}{u_1,u_2,\ldots,u_z} \binom{n-m}{m-z}=\frac{1}{m!}\frac{\prod_{j=0}^{z+\rho-1}(m-j)\prod_{j=0}^{m-z-1}(n-m-j)}{\prod_{j=1}^z u_j!}.
\]

We fix a $z$-irregular set $I=\{i_1,i_2,\ldots,i_m\}$ of type $\lambda$, and provide estimates for $\Ex{\X_I}$. Since $I$ is $z$-irregular, $|I\cup(1+I)|=2m-z$. Let $\overline{I}\subseteq\{1,2,\ldots,n\}$ be a set of size $t+2m-z$ containing $I\cup(1+I)$. Define $\nu_I$ to be the number of choices of distinct integers $h_{i'}\in\{1,2,\ldots,n\}$ for $i'\in\overline{I}$ such that $h_{i_j+1}-h_{i_j}\in K$ for $1\leq j\leq m$. Then the number of permutations $\pi$ with $\pi(i_j+1)-\pi(i_j)\in K$ for $1\leq j\leq m$ is $ (n-t-2m+z)! \nu_I$, and so
\[
\Ex{\X_I}=\frac{ (n-t-2m+z)! \nu_I}{n!}=\frac{\nu_I}{\prod_{j=0}^{t+2m-z-1}(n-j)}.
\]

We will use the following lemma, proved in the next section, to provide good estimates for~$\nu_I$:

\begin{lemma}
\label{lem:final_counting}
The integer $\nu_I$ only depends on $d$, $m$ and the type $\lambda$ of $I$. For a partition $\lambda$, we may therefore define $\mu(\lambda)$ (a function of $d$ and $m$) by $\mu(\lambda):=\nu_I$ where $I$ is any set of type $\lambda$.

Let $\lambda$ be a fixed partition of an integer $z$ with $\rho$ parts. Then for all sufficiently large~$n$, and all $m$ such that $m\geq z+\rho$,
\[
\mu(\lambda)=(2d)^m\left(n^{t+m}\left(\sum_{i=0}^{t-1}p_i^\lambda/n^i+O(q^\lambda/n^t)\right)\right)
\]
where $p_0^\lambda,p_1^\lambda,\ldots,p_{t-1}^\lambda,q^\lambda$ are $d$-SP($m$) functions such that
$p_i^\lambda=0$ when $i<z$. Moreover, when $\lambda$ is the empty partition (so $z=0$), then $p_0^{\lambda}=1$.
\end{lemma}

Using this lemma, and summing over partitions $\lambda$ of $z$, we find that whenever $m\geq z+\rho$
\begin{align*}
\sum_{\substack{|I|=m\\
\text{$z$-irregular}}}\Ex{\X_I}&=\sum_\lambda \frac{1}{m!}\frac{\prod_{j=0}^{z+\rho-1}(m-j)\prod_{j=0}^{m-z-1}(n-m-j)}{\prod_{j=0}^{t+2m-z-1}(n-j)\prod_{j=1}^z u_j!} \mu(\lambda)\\
&=\sum_\lambda \frac{(2d)^m}{m!}\frac{n^{t+m}\prod_{j=0}^{z+\rho-1}(m-j)\left(\sum_{i=0}^{t-1}p_i^\lambda/n^i+O(q^\lambda/n^t)\right)}{\prod_{j=0}^{m-1}(n-j)\prod_{j=0}^{t-1}(n-2m+z-j)\prod_{j=1}^z u_j!},
\end{align*}
where $p_0^\lambda,p_1^\lambda,\ldots,p_{t-1}^\lambda,q^\lambda$ satisfy the conditions of Lemma~\ref{lem:final_counting}. Though we have imposed the restriction that $m\geq z+\rho$, the formula actually holds for all non-negative values of $m$, since for $0\leq m<z+\rho$ the left hand side is an empty sum (there are no $z$-regular subsets of cardinality $m$ with a type having $\rho$ parts), and the right hand side is zero due to the product $\prod_{j=0}^{z+\rho-1}(m-j)$ in the numerator. But the number of possible types $\lambda$ of a $z$-regular set for $z<t$ is finite, and the $t$ most significant coefficients of the product $\prod_{j=0}^{m-1}(n-j)\prod_{j=0}^{t-1}(n-2m+z-j)$ are polynomials in $m$. Hence there exist $d$-SP($m$) functions $\overline{p}_0,\overline{p}_1,\ldots,\overline{p}_{t-1}$ and $q_1$ such that
\begin{equation}
\label{eqn:first_sum}
\sum_{z=0}^{t-1}\sum_{\substack{|I|=m\\
\text{$z$-irregular}}}\Ex{\X_I} = \frac{(2d)^m}{m!}\left((\sum_{i=0}^{t-1}\overline{p}_i/n^i)+O(q_1/n^t)\right).
\end{equation}
Moreover, since $p_0^\lambda=0$ when $z>0$ and $p_0^\lambda=1$ otherwise, we see that $\overline{p}_0=1$.

The equation~\eqref{eqn:first_sum} provides a good estimate for the first sum on the right hand side of equation~\eqref{eqn:split_sum}.  We show that the second sum in~\eqref{eqn:split_sum} is small. Clearly $\nu_I\leq (2d)^m n^{t+m-z}$, as the integers $h_{i'}$ are determined by fixing one of the $(2d)^m$ choices for the differences $h_{i_j+1}-h_{i_j}\in K$ and then choosing the $t+m-z$ integers $h_{i'}\in [n]$ where $i'\in\overline{I}\setminus (1+I)$.
So $\Ex{\X_I}\leq (2d)^m n^{-m}(1+O(m^2/n))=(2d)^mO(n^{-m})$. If $I$ is
$z$-irregular, its run lengths $\ell_1,\ell_2,\ldots ,\ell_{m-z}$ are positive integers summing to $m$, and so the number of possibilities for these run lengths is
\[
\binom{m-1}{m-z-1}=\binom{m-1}{z}\leq m^z.
\]
Once the run lengths are determined, the set $I$ is determined by the lengths of the gaps between runs, and there are (see above) $\binom{n+1-m}{m-z}$ possibilities for these gaps. Since
\[
\binom{n+1-m}{m-z}\leq \frac{n^{m-z}}{(m-z)!}\leq \frac{m^zn^{m-z}}{m!},
\]
we find that
\begin{align*}
\sum_{z\geq t}\sum_{\substack{|I|=m\\
\text{$z$-irregular}}}\Ex{\X_I}&\leq\frac{(2d)^m}{m!}\sum_{z\geq t} m^{2z}n^{m-z}O(n^{-m})\\
&=\frac{(2d)^m}{m!}O(m^{2t}n^{-t}).
\end{align*}

Combining this bound with~\eqref{eqn:first_sum}, and setting $\overline{q}=q_1+m^{2t}$, the estimate~\eqref{eqn:core_estimate} follows, and so our theorem follows once Lemma~\ref{lem:final_counting} is proved.
\end{proof}

\section{Choosing sets of integers}
\label{sec:sets}

In this section, we complete the proof of Theorem~\ref{thm:mj_prob}
 by proving Lemma~\ref{lem:final_counting}. This lemma provides good bounds on the number of vectors of distinct integers from $\{1,2,\ldots,n\}$ subject to some constraints on their relative values. We begin this section by discussing an easier problem, where the integers are not required to be distinct. We phrase the constraints we are interested in as a finite labelled graph. We then consider a more general problem involving distinct integers, before proving the lemma.
 
 \subsection{Integers that are not necessarily distinct}
 
Recall from Section~\ref{sec:jumps} that we define $K$ to be the set of all non-zero integers $\delta$ with $|\delta|\leq d$. Let $\Gamma$ be a simple loopless graph on a vertex set $V$, whose edges are coloured red and blue. We say that $\omega:V\times V\rightarrow K\cup\{0\}$ is a \emph{height labelling} for $\Gamma$ if $\omega(u,v)=-\omega(v,u)$ for all $u,v\in V$ and if $\omega(u,v)\in K$ if and only if $uv$ is a red edge. When $\Gamma$ has $m$ red edges, there are exactly $(2d)^{m}$ height labellings of $\Gamma$, since $|K|=2d$.

Let $n$ be a sufficiently large integer (more precisely, we assume that $n>d|V|$). We write $z_\omega(\Gamma)$ for the number of sequences $(h_v:v\in V)$ of positive integers such that:
\begin{align*}
h_v&\in[n] \text{ for all }v\in V\\
h_u&=h_v+\omega(u,v)\text{ for all edges $uv$ in $\Gamma$ (whether red or blue).}
\end{align*}
For example, if $\Gamma$ is a pair of vertices $u$ and $v$ joined by an edge (either red or blue) then $z_\omega(\Gamma)=n-|\omega(u,v)|$. Another example: Let $\Gamma$ be the triangle on three vertices $u$, $v$ and~$x$, with $uv$ and $vx$ coloured red and $ux$ coloured blue. Then $z_\omega(\Gamma)=0$ when $\omega(u,v)+\omega(v,x)\not=0$, otherwise $z_\omega(\Gamma)=n-|\omega(u,v)|$.

For an $\ell$-vertex walk $p=u_1,u_2,\ldots,u_\ell$ in $\Gamma$, we define the \emph{incline} $\omega(p)$ by
\[
\omega(p)=\sum_{i=1}^{\ell-1} \omega(u_i,u_{i+1}).
\]
In particular, if $p$ is a trivial ($1$-vertex) walk, then $\ell=1$ and $\omega(p)=0$.

We say that a height labelling $\omega$ is \emph{inconsistent} if there is a cycle in $\Gamma$ of non-zero incline; otherwise we say that $\omega$ is \emph{consistent}. Clearly $z_\omega(\Gamma)=0$ if $\omega$ is inconsistent. When $\omega$ is consistent, we may determine $z_\omega(\Gamma)$ as follows. Partition the vertices of $\Gamma$ into connected components $C_1,C_2,\ldots,C_c$. For $i\in\{1,2,\ldots,c\}$ let $\Gamma_i$ be the edge-coloured graph induced on $C_i$, and let $\omega_i$ be the restriction of $\omega$ to the vertices in $C_i$. We see that
\[
z_\omega(\Gamma)=\prod_{i=1}^c z_{\omega_i}(\Gamma_i).
\]
Now suppose that $\Gamma$ is connected. Fix a vertex $u\in \Gamma$. Let
\[
b_1=|\min\{\omega(p)\}|=-\min\{\omega(p)\},
\]
where the minimum is taken over all walks $p$ starting from $u$. Similarly, let
\[
b_2=|\max\{\omega(p)\}|=\max\{\omega(p)\}.
\]
Since $\omega$ is consistent, we may in fact restrict our attention to paths rather than walks and so $b_1+b_2\leq d|V|$. We note that (since $\Gamma$ is connected) the integers $h_v$ counted by $z_\omega(\Gamma)$ are determined by $h_u$. Moreover, the conditions that $h_v\in[n]$ are equivalent to the condition that $1+b_1\leq h_u\leq n-b_2$. Thus, since $n$ is sufficiently large, $z_\omega(\Gamma)=n-b_1-b_2$ in this case. 

One consequence of the analysis above is that when $\Gamma$ has $c$ components and $\omega$ is consistent we may write
\[
z_\omega(\Gamma)=n^c+\sum_{x=1}^c (-1)^xy_x n^{c-x}
\]
where the coefficients $y_x$ are integers that are bounded in absolute value by $c^x(d|V|)^x$. Indeed, an argument using inclusion-exclusion shows that $n^c+\sum_{x=1}^{t-1} (-1)^xy_x n^{c-x}$ is a lower bound for $z_\omega(\Gamma)$ when $t-1$ is odd, and an upper bound when $t-1$ is even.

We write $Z(\Gamma)=\sum_\omega z_\omega(\Gamma)$, where the sum is over all height labellings of~$\Gamma$. We note the following:
\begin{itemize}
\item When $\Gamma$ isolated point, $Z(\Gamma)=n$.
\item When $\Gamma$ is a pair of vertices joined by a red edge, then
\[
Z(\Gamma)=\sum_{\delta\in K} (n-|\delta|)=2dn-d(d+1)=2d(n-\tfrac{1}{2}(d+1)).
\]
\item Suppose that $V$ can be written as the disjoint union $V=V_1\cup V_2$, where there are no edges between $V_1$ and $V_2$. Let $\Gamma_1$ and $\Gamma_2$ be the induced edge-coloured subgraphs on $V_1$ and $V_2$ respectively. Then
\[
Z(\Gamma)=Z(\Gamma_1) Z(\Gamma_2).
\]
\end{itemize}

\subsection{Distinct integers}

Let $\Gamma$ be a graph with all edges coloured red, and let $\omega$ be a height labelling of $\Gamma$. We write $z^*_\omega(\Gamma)$ for the number of sequences $(h_v:v\in V)$ of \emph{distinct} positive integers counted by $z_\omega(\Gamma)$, in other words sequences such that:
\begin{align*}
h_v&\in\{1,2,\ldots,n\} \text{ for all }v\in V\\
h_u&=h_v+\omega(u,v)\text{ for all (red) edges $uv$ in $\Gamma$}\\
h_u&\not=h_v\text{ for all $u,v\in V$ with $u\not=v$}.
\end{align*}
Define
\[
Z^*(\Gamma)=\sum_\omega z^*_\omega(\Gamma),
\]
where the sum is over all height labellings $\omega$ of $\Gamma$.
 
Let $E$ be the set of edges in the complement of $\Gamma$. For $F\subseteq E$, let $\Gamma_F$ be the graph obtained by adding the edges in $F$ to $\Gamma$, and colouring these new edges blue. The Principle of Inclusion-Exclusion states that
\[
Z^*(\Gamma)=\sum_{f=0}^{|E|}(-1)^f \sum_{\substack{F\subseteq E\\|F|=f}}Z(\Gamma_F).
\]
Moreover, the partial sum
\[c_0
\sum_{f=0}^a(-1)^f \sum_{\substack{F\subseteq E\\|F|=f}}Z(\Gamma_F)
\]
is an upper bound for $Z^*(\Gamma)$ when $a$ is even, and a lower bound when $a$ is odd.

For a function $f$ of $d$, $m$ and $n$, and a function $g$ of $n$, we write $f=\sO(g)$ to mean that $f=O(qg)$, where $q$ is a polynomial in $d$ and $m$. (Since $d$ and $m$ are logarithmic in $n$ in our applications, this agrees with the usual meaning of `soft Oh'.) 
\begin{lemma}
\label{lem:graph_count}
Let $n_0$ and $t$ be fixed constants. Let $\Gamma_1$ be a graph on $n_0$ vertices, with all edges coloured red, and with $c_0$ components. Let $\Gamma_2$ be a union of $m'$ disjoint red edges for some integer $m'$. Let $\Gamma$ be the disjoint union of $\Gamma_1$ and $\Gamma_2$. Suppose that $\Gamma$ has $m$ red edges. Let $c=c_0+m'$ be the number of components of $\Gamma$. Suppose that $c\geq t$. Then
\begin{equation}
\label{eqn:sp}
Z^*(\Gamma)/(2d)^{m}=\left(\sum_{i=0}^{t-1}x_i n^{c-i}\right)+\sO(n^{c-t}),
\end{equation}
where the coefficients $x_i$ are $d$-SP($m$). When $\Gamma_1$ contains no edges, $x_0=1$.
\end{lemma}
\begin{proof} Let $a$ be the smallest integer with the property that adding $a+1$ blue edges to $\Gamma$ always results in a graph with $c-t$ or fewer components. We note that $a$ is bounded, since $a\leq \binom{n_0+2t}{2}$. We have that
\begin{align*}
\left|Z^*(\Gamma)-\sum_{f=0}^a(-1)^f \sum_{\substack{F\subseteq E\\|F|=f}}Z(\Gamma_F)\right|&\leq \sum_{\substack{F\subseteq E\\|F|=a+1}}Z(\Gamma_F)\\
&\leq \sum_{\substack{F\subseteq E\\|F|=a+1}} (2d)^m n^{c-t}\\
&\leq \binom{n_0+2m}{a+1}(2d)^m n^{c-t}\\
&=(2d)^m\sO(n^{c-t}).
\end{align*}
For $0\leq s\leq 2a$, let ${\cal F}_s$ be the collection of subsets $F\subseteq E$ such that $|F|\leq a$ and such that exactly $s$ components of $\Gamma_2$ contain an extremity of some edge in $F$ (so are no longer components in $\Gamma_F$).  We see that
\begin{align}
\nonumber
\sum_{f=0}^a(-1)^f \sum_{\substack{F\subseteq E\\|F|=f}}Z(\Gamma_F)&=\sum_{s=0}^{2a}\sum_{F\in {\cal F}_s}(-1)^{|F|}Z(\Gamma_F)\\
\label{eqn:half_way}
&=\left(\sum_{s=0}^{t-1}\sum_{F\in {\cal F}_s}(-1)^{|F|}Z(\Gamma_F)\right)+(2d)^m\sO(n^{c-t}),
\end{align}
the last line following since for any $s\geq t$ we have $Z(\Gamma_F)\leq (2d)^m n^{c-t}$ for all $F\in{\cal F}_s$, and since $|{\cal F}_s|=\sO(1)$ (because $|{\cal F}_s|\leq \binom{|E|}{s}$ where $|E|=\binom{n_0+2m}{2}$). There are a bounded number of possibilities for $s$ in the sum~\eqref{eqn:half_way}, so the statement~\eqref{eqn:sp} will follow provided we can show that for any fixed value of $s$
\begin{equation}
\label{eqn:finish}
\sum_{F\in {\cal F}_s}(-1)^{|F|}Z(\Gamma_F)=(2d)^{m}\left(\sum_{i=0}^{t-1}x'_i n^{c-i}\right)+(2d)^{m}\sO(n^{c-t}).
\end{equation}
where the coefficients $x'_i$ are $d$-SP($m$).

Let $\Gamma'$ be the disjoint union of $\Gamma_1$ and a graph $\Gamma'_2$ which is the disjoint union of $s$ red edges. For a subset $F'$ of edges in the complement of $\Gamma'$, we define $\Gamma'_{F'}$ in the same way as $\Gamma_F$. Define ${\cal F}'$ to be the collection of those sets of edges $F'$ with endpoints in all components in $\Gamma'_2$, so all the components of $\Gamma'_2$ are merged into larger components in $\Gamma'_{F'}$. We note that for $F\in{\cal F}_s$, the graph $\Gamma_F$ is isomorphic to the disjoint union of $\Gamma'_{F'}$ for some $F'\in{\cal F}'$ and $m'-s$ disjoint red edges. Moreover,
\[
Z(\Gamma_F)=Z(\Gamma'_{F'})\,(n-\tfrac{1}{2}(d+1))^{m'-s}.
\]
Since there are $\binom{m'}{s}=\binom{m-c_0}{s}$ choices for the components of $\Gamma_2$ merged by $F\in{\cal F}_s$, we see that
\[
\sum_{F\in {\cal F}_s}(-1)^{|F|}Z(\Gamma_F)=\binom{m-c_0}{s}(n-\tfrac{1}{2}(d+1))^{m'-s}\left(\sum_{F'\in{\cal F}'}(-1)^{|F'|}Z(\Gamma'_{F'})\right).
\]
But the sum $\sum_{F'\in{\cal F}'}(-1)^{|F'|}Z(\Gamma'_{F'})$ is a polynomial in $n$ of degree at most $c$, whose coefficients are $d$-SP($m$) since $\Gamma'$ has bounded order. (Indeed, the coefficients do not, in fact, depend on $m$ at all.) Moreover, $\binom{m-c_0}{s}$ is a polynomial in $m$ since $m-c_0=m'\geq 0$. Hence \eqref{eqn:finish} holds, and so the theorem main statement of the lemma holds.

For the last statement of the lemma, we note that $Z(\Gamma_\emptyset)/(2d)^m$ is a monic polynomial when $n_0=0$, since in this case $Z(\Gamma_\emptyset)/(2d)^m=n^{c_0}(n-\frac{1}{2}(d+1))^{m'}$. Moreover, $\Gamma_F$ has fewer than $c$ components when $F$ is non-empty, so $Z(\Gamma_\emptyset)$ is a polynomial of degree strictly less than~$c$. 
\end{proof}

\subsection{A proof of Lemma~\ref{lem:final_counting}}

Recall the notation used in the statement of Lemma~\ref{lem:final_counting}. We prove this lemma as follows.

Let $\lambda=1^{u_1}2^{u_2}\cdots z^{u_z}$ be a partition of $z$ with $\rho$ parts. Let $I$ be a $z$-regular $m$-subset of $\{1,2,\ldots,n-1\}$ of type $\lambda$. Note that such subsets $I$ exist, since we are assuming that $m\geq z+\rho$. For $j\geq 1$, the subset $I$ has exactly $u_j$ runs of length $j+1$, and $m-z-\rho$ runs of length $1$. Moreover, $|I\cup(1+I)|=2m-z$.

Let $\overline{I}$ be a subset of $\{1,2,\ldots,n\}$ of cardinality $t+2m-z$ containing $I\cup(1+I)$. Define an edge-coloured graph $\Gamma$ on the vertex set $\overline{I}$ by adding a red edge between $i$ and $i+1$ for each $i\in I$. We see that $\nu_I = Z^*(\Gamma)$.

The graph $\Gamma$ has $m$ edges. It has $t$ isolated vertices, namely the vertices in $\overline{I}\setminus (I\cup (1+I))$. The induced graph on the remaining vertices is a union of disjoint red paths, with $m-z-\rho$ paths of length $1$ and $u_j$ paths of length $j+1$ for $j\geq 1$. In particular, the isomorphism class of $\Gamma$ depends only on $m$ and $\lambda$. This proves the first statement of the lemma.

Let $\Gamma_2$ be the graph consisting of the $m-z-\rho$ isolated edges, so $\Gamma_2$ has $2(m-z-\rho)$ vertices. Let $\Gamma_1$ be the graph induced on the remaining vertices. We note that $\Gamma_1$ has $n_0:=t+z+2\rho$ vertices: this is a constant, since $t$ and $\lambda$ are fixed. The graph $\Gamma$ has $c=t+m-z$ components, since it has $t$ isolated vertices, $m-z-\rho$ isolated edges, and $\rho$ disjoint paths of length $2$ or more. In particular, $c\geq t$ and so Lemma~\ref{lem:graph_count} applies. The main statement of the lemma now follows from Equation~\eqref{eqn:sp} in Lemma~\ref{lem:graph_count}, since $c\leq t+m$. Since $c=t+m-z$, the statement that $p_i^\lambda=0$ when $i<z$ also follows from~\eqref{eqn:sp}, since no power of $n$ higher than $c$ appears on the right hand side of~\eqref{eqn:sp}.

Now suppose that $z=0$, so $\lambda$ is the empty partition. In this case, $\Gamma_1$ consists of $t$ isolated points, and $\Gamma$ has exactly $c=t+m$ components. The final statement of the lemma now follows from the final statement of Lemma~\ref{lem:graph_count}.\hfill\qed

\section{Expectation and higher moments}
\label{sec:expectation}

We are now in a position to establish the asymptotic moments of $\br(\pi)$ and $\mj(\pi)$. We use the notation defined in the introduction.

\begin{proof}[Proof of Theorem~\ref{thm:expectedbreadth}]  Define $F_\ell(n)=\Pr{\br(\pi)\geq \ell}$, where $\pi\in\Sn$ is chosen uniformly, so $F_\ell(n)=\Pr{\mathbf{Y}_n\geq \ell}$. Define $F_\ell(\infty)=\Pr{\mathbf{Y}\geq \ell}$.

We claim that, for any non-negative integer $u$,
\begin{equation}
\label{eqn:breadth_moment_claim}
\sum_{d=0}^{\infty} (d+2)^uF_{d+2}(n)\rightarrow \sum_{d=0}^{\infty} (d+2)^uF_{d+2}(\infty)
\end{equation}
as $n\rightarrow\infty$. Once this claim is established, Theorem~\ref{thm:expectedbreadth} may be proved as follows. Since $\br(\pi)\geq 2$ for any permutation $\pi\in\Sn$, the $a$th moment of $\mathbf{Y}_n$ may be written as
\[
\sum_{d=0}^\infty (d+2)^a\Pr{\mathbf{Y}_n= d+2}=1+\sum_{d=0}^\infty \big((d+2)^a-(d+1)^a\big)F_{d+2}(n).
\]
Now, $(d+2)^a-(d+1)^a=\sum_{i=0}^{a-1}v_i(d+2)^i$ for some integers $v_i$. So, using our claim~\eqref{eqn:breadth_moment_claim},
\begin{align}
1+\sum_{d=0}^\infty \big((d+2)^a-(d+1)^a\big)F_{d+2}(n)&=1+\sum_{i=0}^{a-1}\left(v_i \sum_{d=0}^\infty (d+2)^iF_{d+2}(n)\right)\nonumber\\
&\rightarrow 1+\sum_{i=0}^{a-1}\left(v_i \sum_{d=0}^\infty (d+2)^iF_{d+2}(\infty)\right)\label{eqn:special_case}\\
&=\sum_{d=0}^\infty (d+2)^a\Pr{\mathbf{Y}= d+2},\nonumber
\end{align}
and this final expression is the $a$th moment of $\mathbf{Y}$. This establishes the first statement of the theorem. The final statement of the theorem follows from~\eqref{eqn:special_case} in the case when $a=1$, since $v_0=1$ in this case and since $F_{d+2}(\infty)=e^{-d^2-d}$. It remains to prove the claim.

Define $\kappa=\lceil(1/2)\sqrt{\log(n)}\rceil$, and define $\kappa'=\lceil (\frac{1}{2}u+1)\log n\rceil$. Bevan et al~\cite{cheyne:prolific} showed that when $\pi\in \Sn$ then $n\geq \lceil \br(\pi)^2/2+2\br(\pi)+1\rceil$. In particular, this shows that $\br(\pi)\leq \lceil \sqrt{2n}\,\rceil+2$. Hence
\[
\sum_{d=0}^{\infty} (d+2)^uF_{d+2}(n)=\sum_{d=0}^{\lceil \sqrt{2n}\rceil} (d+2)^uF_{d+2}(n)=s_1+s_2+s_3,
\]
where
\begin{align*}
s_1&=\sum_{d=0}^{\kappa} (d+2)^uF_{d+2}(n)\\
s_2&=\sum_{d=\kappa+1}^{\kappa'} (d+2)^uF_{d+2}(n)\text{ and}\\
s_3&=\sum_{d=\kappa'+1}^{\lceil \sqrt{2n}\rceil} (d+2)^uF_{d+2}(n).
\end{align*}
We consider these three sums separately.
A good estimate for $s_1$ follows directly from Theorem~\ref{thm:direct}:
\[
\sum_{d=0}^{\kappa} (d+2)^uF_{d+2}(n)=\left(\sum_{d=0}^\kappa (d+2)^ue^{-d^2-d}\right)+O
\left(\sum_{d=0}^\kappa (d+2)^u\frac{(\log n)^6}{n}e^{d^2+d}\right),
\]
which tends to $\sum_{d=0}^\infty (d+2)^ue^{-d^2-d}$ as $n\rightarrow\infty$. So the theorem follows once we can prove that $s_2$ and $s_3$ both tend to zero as $n\rightarrow\infty$.

Now $s_2$ is non-negative, and we see that
\[
s_2\leq \kappa' (\kappa'+2)^uF_{\kappa+2}(n)
\]
since there are at most $\kappa'$ terms in the sum, and since $F_{d+2}(n)$ is a decreasing function of $d$. But
\[
F_{\kappa+2}(n)\leq e^{-\kappa^2-\kappa}+O\left(\frac{(\log n)^6}{n}e^{\kappa^2+\kappa}\right)=o(n^{-1/8}).
\]
Thus, since $\kappa'(\kappa'+2)^u=O((\log n)^{u+1})$, we find that $s_2\rightarrow 0$ as required.

The above method is not sufficient to show that $s_3\rightarrow 0$, as there are about $\sqrt{n}$ terms in the sum $s_3$ and since either the leading term or the error term in our upper bound for $F_{d+2}(n)$ always grows faster than $n^{-1/2}$. For the range of values of $d$ we are interested in, we expect the leading term to be small and so we are willing to sacrifice the tightness of the leading term in order to reduce the error term. 

Now, $s_3\leq \sqrt{2n}(\sqrt{2n}+3)^uF_{\kappa'+2}(n)\leq (25n)^{\frac{1}{2}(u+1)}F_{\kappa'+2}(n)$. 
We note that for any permutation $\pi$ we have $\br(\pi)\leq \mj(\pi)+1$, and so $F_{d+2}(n)$ is bounded above by the probability that $\mj(\pi)\geq d+1$. Thus, by Theorem~\ref{thm:mj_prob}, there exist polynomially bounded functions $p_1(x),p_2(x),\ldots,p_{3u+3}(x)$ and $q(x)$ such that
\[
s_3\leq (25n)^{\frac{1}{2}(u+1)}\left(1+\sum_{i=1}^{3u+3}p_i(\kappa')/n^i\right) e^{-2\kappa'}+ (25n)^{\frac{1}{2}(u+1)}(q(\kappa')/n^{3u+4})e^{2\kappa'}.
\]
The right hand side of this inequality tends to $0$, since $n^{\frac{1}{2}(u+1)}e^{-2\kappa'}=o(n^{-1})$ and since $n^{\frac{1}{2}(u+1)-(3u+4)}e^{2\kappa'}=o(n^{-1})$.
Thus, since $s_3$ is non-negative, $s_3\rightarrow 0$ as required.
\end{proof}

\begin{proof}[Proof of Theorem~\ref{thm:expectedjump}]
The proof is similar to the proof of Theorem~\ref{thm:expectedbreadth}.  Define $G_\ell(n)=\Pr{\mathbf{Z}_n\geq \ell}$ and define $G_\ell(\infty)=\Pr{\mathbf{Z}\geq \ell}$. We claim that, for any non-negative integer $u$,
\begin{equation}
\label{eqn:jump_moment_claim}
\sum_{d=0}^{\infty} (d+1)^uG_{d+1}(n)\rightarrow \sum_{d=0}^{\infty} (d+1)^uG_{d+1}(\infty)
\end{equation}
as $n\rightarrow\infty$. Once this claim is proved, the theorem can be proved as follows. 
The $a$th moment of $\mathbf{Z}_n$ may be written as
\[
\sum_{d=0}^\infty (d+1)^a\Pr{\mathbf{Z}_n= d+1}=\sum_{d=0}^\infty \big((d+1)^a-d^a\big)G_{d+1}(n).
\]
Using our claim~\eqref{eqn:jump_moment_claim},
\begin{align}
\sum_{d=0}^\infty \big((d+1)^a-d^a\big)G_{d+1}(n)&=\sum_{i=0}^{a-1}\left((-1)^{a-1-i}\binom{a}{i} \sum_{d=0}^\infty (d+1)^iG_{d+1}(n)\right)\nonumber\\
&\rightarrow \sum_{i=0}^{a-1}\left((-1)^{a-1-i}\binom{a}{i} \sum_{d=0}^\infty (d+1)^iG_{d+1}(\infty)\right)\label{eqn:special_case2}\\
&=\sum_{d=0}^\infty (d+1)^a\Pr{\mathbf{Z}= d+1},\nonumber
\end{align}
which is the $a$th moment of $\mathbf{Z}$. This establishes the first statement of the theorem. The final statement of the theorem follows from~\eqref{eqn:special_case2} in the case when $a=1$. We will now prove the claim.

Set $\kappa=\lceil (u+1)\log n\rceil$. Since $1\leq\mj(\pi)\leq n$, we see that
\[
\sum_{d=0}^{\infty} (d+1)^uG_{d+1}(n)=\sum_{d=0}^{n-1}(d+1)^uG_{d+1}(n)=s_1+s_2,
\]
where
\begin{align*}
s_1&=\sum_{d=0}^\kappa(d+1)^uG_{d+1}(n) \text{ and}\\
s_2&=\sum_{d=\kappa+1}^{n-1} (d+1)^uG_{d+1}(n).
\end{align*}
Since $\kappa=O(\log n)$, Theorem~\ref{thm:mj_prob} shows that there exist polynomially bounded functions $p_1(x), p_2(x),\ldots p_{4u+3}(x)$ and $q(x)$ such that
\[
G_{d+1}(n)=\left(1+\sum_{i=1}^{4u+3} p_i(d)/n^i\right)e^{-2d}+O(q(d)e^{2d}/n^{4u+4})
\]
for $d\leq \kappa$. So
\[
s_1=\left(\sum_{d=0}^\kappa (d+1)^ue^{-2d}\right)+o(1)\rightarrow \sum_{d=0}^\infty (d+1)^u e^{-2d}
\]
as $n\rightarrow\infty$. Moreover, since $\Pr{\mj(\pi)\geq d+1}$ is a decreasing function of $d$,
\[
s_2\leq n (n+1)^u G_{\kappa+1}(n)\leq n^{u+1}(1+o(1))e^{-2\kappa}+O(n^{u+1}q(\kappa)e^{2\kappa}/n^{4u+4})=o(1).
\]
Hence the theorem follows.
\end{proof}

\section*{Acknowledgements}

The authors would like to thank the reviewers of an earlier draft for their excellent work, which has significantly improved the paper. In particular, we thank a reviewer for their suggestion to generalise to higher moments in Theorems~\ref{thm:expectedbreadth} and~\ref{thm:expectedjump}.

\end{document}